\documentclass[12pt,a4paper]{article}
\usepackage[latin1]{inputenc}
\usepackage{amsmath}
\usepackage{amsfonts}
\usepackage{amssymb}
\usepackage{graphicx}
\usepackage{amsthm}
\usepackage{hyperref}
\usepackage{indentfirst}
\usepackage{url}

\newtheorem{theorem}{Theorem}[section]
\newtheorem{proposition}[theorem]{Proposition}
\newtheorem{corollary}[theorem]{Corollary}
\newtheorem{lemma}[theorem]{Lemma}
\newtheorem{remark}[theorem]{Remark}

\newtheorem{example}[theorem]{Example}
\newtheorem{question}[theorem]{Question}

\begin{document}
	
	\title{A characterization of the Arf property for quadratic quotients of the Rees algebra}
	
	\author{Alessio Borzì
			\thanks{Dipartimento di Matematica e Informatica, Università degli Studi di Catania.} \thanks{Scuola Superiore di Catania.}}
	
	\maketitle
	
	\begin{abstract}
		We provide a characterization of the Arf property in both the numerical duplication of a numerical semigroup and in a member of a family of quotients of the Rees algebra studied in \cite{barucci2015family}.
	\end{abstract}
	
	\section*{Introduction}
	
	Let $R$ be a Noetherian one-dimensional local domain, $I$ be an ideal of $R$ and $t$ be an indeterminate. Let $R[It] = \bigoplus_{n \in \mathbb{N}}I^nt^n$ be the Rees algebra associated with $R$ and $I$. In \cite{barucci2015family} the authors, looking for a unified approach to the Nagata's idealization and the amalgamated duplication of a ring (see \cite{d2007amalgamated}), studied the following family of quotients of the Rees algebra for every $a,b \in R$
	\[ R(I)_{a,b} = \frac{R[It]}{(t^2+at+b) \cap R[It]}, \]
	showing that the Nagata's idealization is obtained for $a=b=0$ and the amalgamated duplication for $a=1$, $b=0$. A remarkable fact about this family of rings is that we can always find domains among its members, if the original ring $R$ is itself a domain. In particular, it was shown in \cite{barucci2015family} that this ring construction can be connected to a semigroup construction called numerical duplication (see \cite{d2013numerical}). More precisely let $S$ be a numerical semigroup, $E$ be a semigroup ideal, and $m \in S$ be an odd integer. For any set of integers $A \subseteq \mathbb{Z}$ we set $2 \cdot A = \{ 2a : a \in A \}$. Then we define the numerical duplication $S \Join^m E$ of $S$ with respect to $E$ and $m$ as the numerical semigroup
	\[ S \Join^m E = 2 \cdot S \cup (2 \cdot E + m). \]
	Now, if we start with an algebroid branch $R$ and $b \in R$ with $v(b)$ odd, the member of the family of the type $R(I)_{0,-b}$ has its value semigroup equal to the numerical duplication of $v(R)$ with respect to $v(I)$ and $v(b)$.	In this paper we show that this is true in general for every Noetherian, one-dimensional, analytically irreducible, local domain $R$. 
	
	In \cite{arf1948interpretation} Arf solved the classification problem of singular branches, using their multiplicity sequence. Later, inspired by the work of Arf, Lipman in \cite{lipman1971stable} introduced the notions of Arf ring and Arf closure of a ring. These rings share the same multiplicity sequence. Hence the idea is to calculate the Arf closure of the coordinate ring of a curve and then its value semigroup, which is an Arf numerical semigroup, in order to obtain its multiplicity sequence. \\
	
	In this paper we provide a characterization of the Arf property in both the numerical duplication and the family of rings $\mathcal{R} = R(I)_{0,-b}$. More precisely, in Section \ref{section 1} we recall all the basic notions on numerical semigroups and Arf rings. In Section \ref{section 2}, we prove the characterization of the Arf property for the numerical duplication (Theorem \ref{characterization on ns}). In Section \ref{section 3}, we show that $\mathcal{R}$ is a Noetherian, one-dimensional, analytically irreducible, local domain and its value semigroup is $v(R) \Join^{v(b)}v(I)$ (Theorem \ref{double valuation}), then we prove a series of technical results for the purpose of proving Theorem \ref{characterization on rings}, that is the extension to $\mathcal{R}$ of the previous numerical characterization. \\
	
	Several computations are performed by using the GAP system \cite{gap2015gap} and, in
	particular, the NumericalSgps package \cite{delgadonumericalsgps}.
	
	\section{Preliminaries}\label{section 1}
	
	A numerical semigroup $S$ is an additive submonoid of $\mathbb{N}$ with finite complement in $\mathbb{N}$. The \emph{multiplicity} of $S$ is $\mu(S) = \min (S \setminus \{0\})$. The \emph{Frobenius number} of $S$ is $F(S) = \max(\mathbb{N} \setminus S)$ and the \emph{conductor} of $S$ is $c(S) = F(S)+1$. A semigroup ideal of $S$ is a subset $E \subseteq S$ such that $E+S \subseteq E$. We call $\tilde{e} = \min E$; then the \emph{integral closure} of $E$ in $S$ is the semigroup ideal $\overline{E} = \{ s \in S: s \geq \tilde{e} \}$, if $E = \overline{E}$ then $E$ is \emph{integrally closed}. We say that $E$ is \emph{stable} if $E+E = E+\tilde{e}$.
	
	An \emph{Arf numerical semigroup} is a numerical semigroup $S$ in which for every $x,y,z \in S$, such that $x \geq y \geq z$, it results $x+y-z \in S$; equivalently $S$ is Arf if and only if every integrally closed semigroup ideal is stable (see \cite[Theorem 2.2]{lipman1971stable}).
	
	Given an Arf numerical semigroup $S = \{ 0 = s_0 < s_1 < s_2 < \ldots \}$ the sequence $(e_0,e_1,e_2,\ldots)$, with $e_i = s_{i+1}-s_i$, is the \emph{multiplicity sequence} of $S$. Note that $e_0$ corresponds to the multiplicity of $S$.
	
	We call an \emph{Arf sequence} a non-increasing sequence of positive integers $(e_0,e_1,e_2,\ldots)$ such that
	\begin{enumerate}
		\item exists $n \in \mathbb{N}$ such that $e_k = 1$ for all $k \geq n$,
		\item for every $i \in \mathbb{N}$ exists $k \geq 1$ such that $e_i = \sum_{j=1}^k e_{i+j}$. 
	\end{enumerate}
	A sequence of positive integers $(e_0,e_1,e_2,\ldots)$ is an Arf sequence if and only if it is a multiplicity sequence of an Arf numerical semigroup, that is $S = \{ 0,e_0,e_0+e_1,e_0+e_1+e_2,\ldots \}$, see for instance \cite[Proposition 1]{garcia2017parametrizing}.
	
	From the Arf numerical semigroup $S$ we can construct a chain of Arf numerical semigroups $S_0 \subseteq S_1 \subseteq S_2 \subseteq \ldots$ with $S=S_0$ and $S_{i+1} = (S_i \setminus \{0\})-e_i$, namely the blow up of $S_i$. The multiplicity of $S_i$ is $e_i$ and $S_n = \mathbb{N}$ for $n$ large enough.
	
	Let $(R,\mathfrak{m})$ be a Noetherian one-dimensional local domain and $\overline{R}$ its integral closure in its field of fractions $Q(R)$.
	We assume that $R$ is \emph{analytically irreducible}, that is its completion $\hat{R}_{\mathfrak{m}}$ is a domain, or, equivalently, $\overline{R}$ is a discrete valuation ring (DVR) and a finitely generated $R$-module.
	Since the integral closure $\overline{R}$ is a DVR, every non zero element of $R$ has a value as an element of $\overline{R}$. The set of values $v(R)$ is a numerical semigroup. The multiplicity of $R$ is equal to the multiplicity of its value semigroup $\mu(R) = \mu(v(R))$.
	
	For any two $R$-submodules $E,F$ of $\overline{R}$ set
	\[ (E:F) = \{ x \in \overline{R}: xF \subseteq E \}. \]
	The blow up of $R$ is $L(R) = \bigcup_{n \in \mathbb{N}}(\mathfrak{m}^n : \mathfrak{m}^n)$. If we fix $R = R_0$ and $R_{i+1} = L(R_i)$ then the multiplicity sequence of $R$ is the sequence $(\mu(R_0),\mu(R_1),\ldots)$.
	
	We will also assume that $R$ is \emph{residually rational}, namely its residue field $k = R/\mathfrak{m}$ is isomorphic to the residue field of $\overline{R}$. With this assumption, for any $x,y \in R$ such that $v(x) = v(y) $ there exists an invertible element $u \in R$ such that $v(x-uy) > v(x) = v(y)$. Furthermore, for any fractional ideals $I,J$ of $R$ such that $J \subseteq I$, it results $\lambda(I/J) = |v(I) \setminus v(J)|$.
	
	An element $x \in R$ is said to be integral over the ideal $I$ if $x$ satisfies a relation
	\[ x^n + a_1x^{n-1} + \ldots + a_{n-1}x+a_n = 0 \]
	with $a_j \in I^j$ for $j = 1,2,\ldots,n$. The set $\overline{I}$ of all elements of $R$ which are integral over $I$ is an ideal of $R$ (see \cite{lipman1971stable} or \cite{huneke2006integral}), called the integral closure of $I$ in $R$.
	
	In our setting, the integral closure of an ideal $I$ is equal to
	\[ \overline{I} = I \overline{R} \cap R = \{ x \in R: v(x) \geq \min v(I) \}, \]
	(see \cite[Proposition 1.6.1, Proposition 6.8.1]{huneke2006integral}).
	If $I = \overline{I}$ then $I$ is \emph{integrally closed}. It follows that the ideal $I$ is integrally closed if and only if the semigroup ideal $v(I)$ is integrally closed. The necessity easily follows from the definitions. For the sufficiency we have $\lambda(\overline{I}/I) = |v(\overline{I}) \setminus v(I)| = 0$, then $I = \overline{I}$. Notice that here the assumption that $R$ is residually rational is needed.
	
	If $I$ and $J$ are two ideals of $R$, then either $\min v(I) \geq \min v(J)$ or $\min v(I) < \min v(J)$. If $I$ and $J$ are integrally closed, from the previous observations it follows that either $I \subseteq J$ or $J \subsetneq I$.
	
	The \emph{conductor} of $R$ is $C = (R:\overline{R})$. The conductor is an ideal both of $R$ and $\overline{R}$. The conductor of $v(R)$ is $\min v(C)$. Note that $\overline{R}C = C$, so $v(C) = v(\overline{R}C) = v(C)+\mathbb{N} = \min v(C) + \mathbb{N}$, therefore $C$ is integrally closed. In particular, if $J$ is an integrally closed ideal of $R$, then either $J \subseteq C$ or $C \subsetneq J$. We will frequently use this fact in Section \ref{section 3}.
	
	Let $x \in I$ be such that $v(x) = \min v(I)$. The ideal $I$ is \emph{stable} if $I^2 = xI$, or, equivalently, if $(I:I) = x^{-1}I$. The ring $R$ is an \emph{Arf ring} if every integrally closed ideal is stable. The ring $R$ is Arf if and only if its value semigroup $v(R)$ is an Arf numerical semigroup and the multiplicity sequence of $R$ coincides with the multiplicity sequence of $v(R)$ (see \cite[Theorem II.2.13]{barucci1997maximality}).
	
	\begin{proposition}\label{stable in conductor}
		Let $J$ be an integrally closed ideal of $R$. If $J \subseteq C$ then $J$ is stable.
	\end{proposition}
	\begin{proof}
		Let $x \in J$ such that $v(x) = \min v(J)$, we show that $J^2 = xJ$. Clearly $xJ \subseteq J^2$. Let $i,j \in J$ then
		\[ v(j) \geq v(x) \Rightarrow v(j)-v(x) = v(jx^{-1}) \geq 0 \Rightarrow jx^{-1} \in \overline{R}. \]
		Since $i \in J \subseteq C$, it results $i(jx^{-1}) \in R$, so
		\[ v(ijx^{-1}) = v(i) + v(jx^{-1}) \geq v(i) \geq v(x) \Rightarrow ijx^{-1} \in J, \]
		therefore $ij \in xJ$.
	\end{proof}
	
	\begin{proposition}\label{JC c xJ}
		Let $J$ be an integrally closed ideal of $R$. If $C \subseteq J$ and $x \in J$ is an element of minimum value in $J$ then $JC \subseteq xJ$.
	\end{proposition}
	\begin{proof}
		Observe that $x^{-1}J \subseteq \overline{R}$, hence
		\[ JC = x(x^{-1}J)C \subseteq x\overline{R}C = xC \subseteq xJ. \qedhere \]
	\end{proof}
	
	\section{Arf property in the numerical duplication}\label{section 2}
	
	In this section $S = \{ 0 = s_0 < s_1 < s_2 < \ldots \}$ will be a numerical semigroup, $E$ a semigroup ideal of $S$ and $m \in S$ an odd integer. Recall that the quotient of $S$ by a positive integer $d$ is
	\[ \frac{S}{d} = \{ x \in \mathbb{N}: dx \in S \}. \]
	
	\begin{proposition}\label{Arf quotient}
		For every $d>0$, if $S$ is Arf so is $\frac{S}{d}$.
	\end{proposition}
	\begin{proof}
		Let $x,y,z \in \frac{S}{d}$ with $x \geq y \geq z$, then we have $dx,dy,dz \in S$ with $dx \geq dy \geq dz$ and since $S$ is Arf it follows that
		\[ d(x+y-z) = dx+dy-dz \in S, \]
		hence $x+y-z \in \frac{S}{d}$.
	\end{proof}
	
	By definition of numerical duplication it is clear that $(S \Join^m E) / 2 = S$, hence we immediately get the following
	
	\begin{corollary}\label{Quotient duplication Arf}
		If $S \Join^m E$ is Arf so is $S$.
	\end{corollary}
	
	\begin{lemma}\label{E integrally closed}
		If $S \Join^m E$ is Arf then $E$ is integrally closed.
	\end{lemma}
	\begin{proof}
		Suppose by contradiction that $E$ is not integrally closed. Then there exists $i \in \mathbb{N}$ such that $s_i \in E$ and $s_{i+1} \notin E$. Consider $2s_{i+1},2s_i+m, 2s_i \in S \Join^m E$, since $2s_{i+1} \geq 2s_i$, $2s_i+m \geq 2s_i$ and $S \Join^m E$ is Arf, we have
		\[ 2s_{i+1}+2s_i+m - 2s_i = 2s_{i+1}+m \in S \Join^m E, \]
		which means $s_{i+1} \in E$, contradiction.
	\end{proof}
	
	Let $S$ be an Arf numerical semigroup, and let $(e_0,e_1,e_2,\ldots)$ be its multiplicity sequence. Fix $n \in \mathbb{N}$ to be the smallest integer such that $e_k = 1$ for every $k \geq n$. We recall that $e_0 = \min S \setminus \{0\} = s_1$, and that $s_ {i+1} = e_0+\ldots+e_i$, in particular $s_{n+1} = s_n + 1$ and $s_n = c(S)$ is the conductor of $S$.
	
	In the proof of the following result we will use the fact that for an Arf numerical semigroup $S$, if $x,x+1 \in S$ then $x+\mathbb{N} \subseteq S$ (see for instance \cite[Lemma 11]{rosales2004arf}).
	
	\begin{theorem}\label{characterization on ns}
		The numerical duplication $D = S \Join^m E$ is Arf if and only if $S$ is Arf with multiplicity sequence $(e_0,e_1,\ldots,e_{n-1},1,1,\ldots)$, $e_{n-1} \neq 1$, $E$ is integrally closed and, if $\min(E) < c(S)$, $e_0 = e_1 = \ldots = e_{n-1} = m$.
	\end{theorem}
	
	\begin{proof}
		\emph{Necessity}. From Corollary \ref{Quotient duplication Arf} and Lemma \ref{E integrally closed}, $S$ is Arf and $E$ is integrally closed.
		Now if $\min(E) < c(S) = s_n$ then $s_{n-1} \in E$. Suppose that $m \geq 2e_{n-1} = 2(s_n - s_{n-1})$, then
		\[ 2s_{n-1} + m \geq 2s_n. \]
		Since $2s_{n-1}+m+1$ is even and $s_n$ is the conductor of $S$ we obtain $2s_{n-1}+m+1 = 2s_k$ for some $k \in \mathbb{N}$. Setting $x = 2s_{n-1}+m$ we have $x,x+1 \in D$ which is Arf, hence $x+\mathbb{N} \subseteq D$, and so $x+2 \in D$; this means
		\begin{gather*}
		2s_{n-1}+m+2 = 2(s_{n-1}+1)+m \in D \Rightarrow \\
		\Rightarrow s_{n-1}+1 \in S \Rightarrow s_n = s_{n-1}+1 \Rightarrow \\
		\Rightarrow e_{n-1} = s_n-s_{n-1} = 1,
		\end{gather*}
		which is a contradiction. Therefore $m < 2e_{n-1}$, hence $2s_{n-1} + m < 2s_{n}$, since $D$ is Arf, this implies
		\[ 2s_n+2s_n-(2s_{n-1}+m) = 2s_n + 2e_{n-1}-m \in D. \]
		Furthermore $2s_n + 2e_{n-1}-m$ is odd and
		\[ 2s_n + 2e_{n-1}-m \geq 2s_{n} > 2s_{n-1}+m. \]
		It follows that
		\begin{gather*}
			2s_n + 2e_{n-1}-m \geq 2s_n+m \\
			\Rightarrow m \leq e_{n-1} \leq e_0.
		\end{gather*}
		Since $m \in S$ and it is odd, we must have $m = e_0 = e_1 = \ldots = e_{n-1}$. \\
		\emph{Sufficiency}. If $\min(E) \geq s_n$, then $E = x+\mathbb{N}$ with $x \geq s_n$, so it results $D = 2S \cup ((2x+m)+\mathbb{N})$ and it is easy to check that $D$ is Arf. \\
		Otherwise if $\min(E) < s_n$ and $m=e_0=e_1=\ldots=e_{n-1}$, then $S = e_0\mathbb{N} \cup (ne_0 + \mathbb{N})$ and $E = \{ ie_0,(i+1)e_0,\ldots,(n-1)e_0 \} \cup (ne_0+\mathbb{N})$ for some $i \leq n$. Hence, if $D = \{ 0 = d_0 < d_1 < \ldots < d_k < \ldots \}$ then after some easy calculations it results
		\[ ( d_{k+1}-d_k : k \in \mathbb{N} ) = ( \underbrace{2e_0,2e_0,\ldots,2e_0}_{i\text{ times}}, \underbrace{e_0,e_0,\ldots,e_0}_{2(n-i) \text{ times}}, \underbrace{2,2,\ldots,2}_{\frac{e_0-1}{2} \text{ times}},1,\ldots ), \]
		which is an Arf sequence, so $D$ is Arf.
	\end{proof}
	
	\begin{example}
		Let $S = \langle 3,7,8 \rangle = \{ 0,3,6,\rightarrow \}$, $S$ is Arf and its multiplicity sequence is $(3,3,1,\ldots)$, so $n = 2$. Let $E=S \setminus \{0\}$ and $m = 3$, $E$ is integrally closed, $\min(E) = 3 < 6 = s_n$ and $m = e_0 = e_1$. The numerical duplication is
		\[ S \Join^m E = \langle 6,9,14,16,17,19 \rangle = \{ 0,6,9,12,14,\rightarrow \}, \]
		and it is an Arf numerical semigroup.
	\end{example}
	
	\begin{remark}\label{remark multiples of m}
		In the case $\min(E) < c(S)$ of Theorem \ref{characterization on ns}, the elements of $D= S \Join^m E$ smaller than the conductor of $D$ are multiples of $m$, so they are of the form $k m$ for some $k \in \mathbb{N}$.
	\end{remark}
	
	Recall that the Arf closure $\operatorname{Arf}(S)$ of a numerical semigroup $S$ is the smallest Arf numerical semigroup that contains $S$ (see \cite{rosales2004arf}). Let $\tilde{E}$ be the integral closure in $\operatorname{Arf}(S)$ of the ideal generated by $E$ in $\operatorname{Arf}(S)$. More explicitly, if $\tilde{e} = \min E$, then $\tilde{E} = \{ s \in \operatorname{Arf}(S) : s \geq \tilde{e} \}$.
	
	\begin{proposition}\label{Arf S dup int E}
		With the notation introduced above we have
		\[ \operatorname{Arf}(S) \Join^m \tilde{E} \subseteq \operatorname{Arf}(S \Join^m E). \]
	\end{proposition}
	\begin{proof}
		Since $S = (S \Join^m E)/2 \subseteq \operatorname{Arf}(S \Join^m E)/2$ and, by Proposition \ref{Arf quotient}, $\operatorname{Arf}(S \Join^m E)/2$ is Arf, we got
		\[ \operatorname{Arf}(S) \subseteq \frac{\operatorname{Arf}(S \Join^m E)}{2}; \]
		it follows that $2 \cdot \operatorname{Arf}(S) \subseteq \operatorname{Arf}(S \Join^m E)$. Now if $e \in \tilde{E}$, then $2e \geq 2\tilde{e}$ and $2\tilde{e} + m \geq 2\tilde{e}$, and, since $\operatorname{Arf}(S \Join^m E)$ is Arf and $2e,2\tilde{e},2\tilde{e}+m \in \operatorname{Arf}(S \Join^m E)$, we have
		\[ 2e + 2\tilde{e}+m - 2\tilde{e} = 2e+m \in \operatorname{Arf}(S \Join^m E). \]
		So, $2\tilde{E}+m \subseteq \operatorname{Arf}(S \Join^m E)$ and therefore $\operatorname{Arf}(S) \Join^m \tilde{E} \subseteq \operatorname{Arf}(S \Join^m E)$.
	\end{proof}
	
	Theorem \ref{characterization on ns} gives us sufficient conditions so that the inclusion of Proposition \ref{Arf S dup int E} is an equality. Recall that $\overline{E}$ is the integral closure in $S$ of the semigroup ideal $E$.
	
	\begin{corollary}
		If one of the following conditions holds
		\begin{enumerate}
			\item $\min(E) \geq c(S)$,
			\item $S$ is Arf and $m = e_0 = e_1 = \ldots = e_{n-1}$,
		\end{enumerate}
		then $\overline{E} = \tilde{E}$ and $\operatorname{Arf}(S) \Join^m \overline{E} = \operatorname{Arf}(S \Join^m E)$.
	\end{corollary}
	
	It is easy to see that the previous equality is not true in the general case. In particular, the following example shows that neither the equality
	\[ \operatorname{Arf}(S) = \frac{\operatorname{Arf}(S \Join^m E)}{2} \]
	holds true in general.
	
	\begin{example}
		Let $S= \langle 5,8,11,12,14\rangle = \{ 0,5,8,10,\rightarrow \}$, $E = S \setminus \{0\}$ and $m = 5$. Note that $S$ is Arf, so $S = \operatorname{Arf}(S)$. The numerical duplication of $S$ with respect to $E$ and $m$ is
		\[ S \Join^m E = \{ 0,10,15,16,20,21,22,24,\rightarrow \}. \]
		Since $15,16 \in S \Join^m E$, its Arf closure is $\operatorname{Arf}(S \Join^m E) = \{ 0,10,15\rightarrow \}$; moreover $9 \in \operatorname{Arf}(S \Join^m E)/2$, but $9 \notin \operatorname{Arf}(S) = S$.
	\end{example}
	
	A couple of questions naturally arise.
	
	\begin{question}
		Are there any sufficient and necessary conditions so that the inclusion of Proposition \ref{Arf S dup int E} is an equality? Is there a way to express $\operatorname{Arf}(S \Join^m E)$ in terms of $S$, $E$ and $m$?
	\end{question}
	
	\section{Arf property in $\mathcal{R}$}\label{section 3}
	
	In this section $(R,\mathfrak{m})$ will be a Noetherian, analytically irreducible, residually rational, one-dimensional, local domain with $\operatorname{char}(R) \neq 2$; $v:Q(R) \rightarrow \mathbb{Z}$ will denote the valuation on $Q(R)$ associated to $\overline{R}$. Let $I$ be an ideal of $R$ and let $t$ be an indeterminate; the Rees algebra (also called Blow-up algebra) associated with $R$ and $I$ is the graded subring of $R[t]$ defined as
	\[ R[It] = \bigoplus_{n \in \mathbb{N}}I^nt^n. \]
	Let $b \in R$ such that $v(b) = m$ is odd, we define
	\[ \mathcal{R} = \frac{R[It]}{(t^2-b) \cap R[It]}. \]
	Then $\mathcal{R}$ is a subring of $R[\alpha]$ with $\alpha = t+(t^2-b)$. Furthermore $\mathcal{R}$ and $R[\alpha]$ have the same integral closure $\overline{\mathcal{R}}$ and the same field of fractions $Q(R)[\alpha]$ (see \cite[Corollary 1.8]{barucci2015family}). Since $\alpha$ is integral over $R$, the integral closure $\overline{R}^{Q(R)[\alpha]}$ of $R$ in $Q(R)[\alpha]$ is the same as the integral closure of $\mathcal{R}$. The extension field $Q(R) \subseteq Q(R)[\alpha]$ is finite and since $\operatorname{char}(R) \neq 2$, it is also separable.
	
	\begin{theorem}\label{double valuation}
		The ring $\mathcal{R}$ is a Noetherian one-dimensional local domain analytically irreducible and residually rational. If $v':Q(R)[\alpha] \rightarrow \mathbb{Z}$ is the extension on $Q(R)[\alpha]$ of the valuation of $\overline{\mathcal{R}}$, then $v'_{|Q(R)} = 2v$, $v'(\alpha)=m$ and
		\[ v'(\mathcal{R}) = v(R) \Join^{v(b)} v(I). \]
	\end{theorem}
	\begin{proof}
		From \cite{barucci2015family} we know that if $R$ is Noetherian, one-dimensional and local so is $\mathcal{R}$. Moreover, since $v(b)$ is odd, the polynomial $t^2-b$ is irreducible in $Q(R)[t]$, then, from \cite[Corollary 1.3]{d2017new}, $\mathcal{R}$ is a domain. \\
		Now we prove that $\mathcal{R}$ is analytically irreducible. It is enough to prove that $\overline{\mathcal{R}}$ is local and a finitely generated $\mathcal{R}$-module. Let $x \in \overline{R}$ be an element of valuation $1$, $k =\frac{m-1}{2}$ and $\beta = \frac{\alpha}{x^k} \in Q(R)[\alpha]$; then $\beta^2 = \frac{b}{x^{m-1}} \in \overline{R}$ since $v(\frac{b}{x^{m-1}}) = 1 > 0$, so $\beta$ is integral over $\overline{R}$. We prove that $\overline{\mathcal{R}} = \overline{R}+\overline{R}\beta$. The inclusion $\overline{R}+\overline{R}\beta \subseteq \overline{\mathcal{R}}$ follows from the fact that $\overline{R} \subseteq \overline{\mathcal{R}}$ and that $\beta$ is integral over $\overline{R}$. Conversely, let $p+q\alpha \in \overline{\mathcal{R}} = \overline{R}^{Q(R)[\alpha]}$, where $p,q \in Q(R)$, $q \neq 0$. Then, since $Q(R) \subseteq Q(R)[\alpha]$ is algebraic, from \cite[Theorem 2.1.17]{huneke2006integral} the coefficients $2p$ and $p^2-q^2b$ of the minimal polynomial of $p+q\alpha$ over $Q(R)$ are in $\overline{R}$. In addition, $v(p^2) = 2v(p)$ is even and $v(q^2b) = 2v(q)+m$ is odd, then $v(p^2) \neq v(q^2b)$, therefore
		\begin{gather*}
			0 \leq v(p^2-q^2b) = \min\{ 2v(p),2v(q)+m \} \\
			\Rightarrow
			\begin{cases}
				v(p) \geq 0 \Rightarrow p \in \overline{R} \\
				2v(q)+m \geq 0 \Rightarrow v(q) \geq -k \Rightarrow qx^k \in \overline{R}.
			\end{cases}
		\end{gather*}
		Hence $p + q \alpha = p + qx^k \beta \in \overline{R} + \overline{R}\beta$. Now if we denote by $\overline{\mathfrak{m}}$ the maximal ideal of $\overline{R}$, the ring $\overline{R} + \overline{R}\beta$ is local with maximal ideal $\overline{M} = \overline{\mathfrak{m}} + \overline{R} \beta$; indeed the inverse of $p+q\beta \in \overline{R}+\overline{R}\beta$ with $p \in \overline{R} \setminus \overline{\mathfrak{m}}$, is $\frac{p-q\beta}{p^2-q^2\beta^2}$, in fact $p^2-q^2\beta^2$ is invertible since $0 = v(p^2) \neq v(q^2\beta^2) = 2v(q)+1$ and
		\[ v(p^2-q^2\beta^2) = \min\{v(p^2),v\left(q^2\beta^2\right)\} = v(p^2) = 0.  \]
		It follows that $\overline{\mathcal{R}} = \overline{R} + \overline{R}\beta$ is local.
		
		Now we prove that $\overline{\mathcal{R}}$ is a finitely generated $\mathcal{R}$-module. The field extension $Q(R) \subseteq Q(R)[\alpha]$ is finite and separable, then, by \cite[Theorem 3.1.3]{huneke2006integral}, the integral closure of $\overline{R}$ in $Q(R)[\alpha]$, which is equal to $\overline{R}^{Q(R)[\alpha]} = \overline{\mathcal{R}}$, is a finite module over $\overline{R}$. Now $\overline{R}$ and $\mathcal{R} \simeq R + I \alpha$ are finite modules over $R$, so $\overline{\mathcal{R}}$ is a finite module over $\mathcal{R}$.
		
		Now let $v'(Q(R)) = d\mathbb{Z}$ for some $d \in \mathbb{N}$. It results $v'(x) = d$ and $b \in (x^m) \setminus (x^{m+1})$, namely $b = ux^m$ with $u \in \overline{R}$ invertible. It follows that $v'(b) = mv'(x)$; in addition from $\alpha^2 = b$ we obtain $2v'(\alpha) = v'(b) = mv'(x) = md$. Since
		\[ \begin{split}
			v'(Q(R)[\alpha]) = v'(Q(R) + Q(R) \alpha) = v'(Q(R)) \cup [v'(Q(R)) + v'(\alpha)] = \\ = d\mathbb{Z} \cup (d\mathbb{Z}+v(\alpha)) = \mathbb{Z},
		\end{split} \]
		it must be $d=2$, so $v'_{|Q(R)} = 2v$ and also $v'(\alpha) = m$. It easily follows that
		\[ v'(\mathcal{R}) = v(R) \Join^{v(b)} v(I) = S \Join^m v(I). \]
		
		Finally we show that $\mathcal{R}$ is residually rational. Recall that $\overline{\mathfrak{m}}$ is the maximal ideal of $\overline{R}$ and $\overline{M} = \overline{\mathfrak{m}}+\overline{R}\beta$ is the maximal ideal of $\overline{\mathcal{R}}$. From \cite[Proposition 2.1]{barucci2015family} the maximal ideal of $\mathcal{R}$ is $M = \mathfrak{m}+I\alpha$. Thus
		\[ \overline{\mathcal{R}}/\overline{M} = \frac{\overline{R}+\overline{R}\beta}{\overline{\mathfrak{m}}+\overline{R}\beta} \simeq\overline{R}/\overline{\mathfrak{m}} \simeq R/\mathfrak{m} \simeq \frac{R+I\alpha}{\mathfrak{m}+I\alpha} = \mathcal{R}/M. \]
	\end{proof}
	
	In the following, $v'$ will denote the extension on $Q(R)[\alpha]$ of the valuation of $\overline{\mathcal{R}}$.
	
	\begin{proposition}\label{R is Arf}
		If $\mathcal{R}$ is Arf, so is $R$.
	\end{proposition}
	\begin{proof}
		Let $J$ be an integrally closed ideal of $R$ and let $x \in J$ such that $v(x) = \min v(J)$. Fix $\tilde{J} = \{ y \in \mathcal{R} : v'(y) \geq v'(x) \}$, $\tilde{J}$ is an integrally closed ideal of $\mathcal{R}$, so it is stable, namely $(\tilde{J}:\tilde{J}) = x^{-1}\tilde{J}$. Furthermore, since $\overline{R} \cap \mathcal{R} = R$, we have $J = \tilde{J} \cap R = \tilde{J} \cap \overline{R}$. It suffices to prove that $(J:J) = x^{-1}J$. It is clear that $(J:J) \subseteq x^{-1}J$, in fact, if $j \in (J:J)$, then by definition $xj \in J$, so $j \in x^{-1}J$. Conversely let $j \in x^{-1}J$ and $j' \in J$, we have
		\[ 
		\left. \begin{array}{r}
		j \in x^{-1}J \subseteq x^{-1}\tilde{J} = (\tilde{J}:\tilde{J}) \\
		j' \in J \subseteq \tilde{J}
		\end{array} \right\} \Rightarrow jj' \in \tilde{J}.
		\]
		Further
		\[ 
		\left. \begin{array}{r}
		j \in x^{-1}J \subseteq \overline{R} \\
		j' \in J \subseteq R \subseteq \overline{R}
		\end{array} \right\} \Rightarrow jj' \in \overline{R},
		\]
		it follows that $jj' \in \tilde{J} \cap \overline{R} = J$, so $j \in (J:J)$.
	\end{proof}
	
	Now our aim is to prove the extension of Theorem \ref{characterization on ns} to $\mathcal{R}$, that is $\mathcal{R}$ is Arf if and only if $R$ is Arf, $I$ is integrally closed and a similar condition on the multiplicity sequence. For the necessity, from Proposition \ref{R is Arf} $R$ is Arf, then the multiplicity sequences of $R$ and $v(R)$ coincides. Therefore we can directly apply Theorem \ref{characterization on ns} (see the proof of Theorem \ref{characterization on rings}). For the sufficiency, from Theorem \ref{characterization on ns} we have that $v(R) \Join^{v(b)} v(I)$ is an Arf numerical semigroup, but this is not enough to prove that $\mathcal{R}$ is Arf. In order to prove necessity, we need a series of technical results.
	
	For this purpose we introduce some more notation. We fix an integrally closed ideal $\tilde{J}$ of $\mathcal{R}$; set $J = \tilde{J} \cap R$ and $\tilde{j_1} = \min v'(\tilde{J})$, $j_1 = \min v'(J)$. We denote the conductor of $R$ with $C = (R:\overline{R})$, and the conductor of $\mathcal{R}$ with $C_{\mathcal{R}} = (\mathcal{R}:\overline{\mathcal{R}})$. Note that the inclusion $C_{\mathcal{R}} \cap R \subseteq C$ may be strict. In addition, we will suppose that $C_{\mathcal{R}} \subsetneq \tilde{J}$. \newpage
	
	\begin{proposition}\label{J int closed}
		The ideal $J$ is integrally closed in $R$. Further
		\begin{enumerate}
			\item If $x+y\alpha \in \tilde{J}$ then $x \in J$ and $y\alpha \in \tilde{J}$.
			\item If $\tilde{j_1}$ is even then $j_1 = \tilde{j_1}$ and there exists $x \in J$ such that $v'(x) = \tilde{j_1}$.
			\item $v'(J) = v'(\tilde{J}) \cap v'(\overline{R})$.
			\item If $\tilde{j_1} < \min v'(I \alpha)$ then $\tilde{j_1}$ is even.
		\end{enumerate}
	\end{proposition}
	\begin{proof}
		We have
		\[ \tilde{J} = \{ x \in \mathcal{R} : v'(x) \geq \tilde{j}_1 \}, \quad J = \tilde{J} \cap R = \{ x \in R: v'(x) \geq \tilde{j}_1 \}, \]
		then $j_1 \geq \tilde{j_1}$; since for any $x,y \in R$ $v(x) \geq v(y)$ if and only if $v'(x) \geq v'(y)$, $\overline{J} \subseteq J$, i.e. $J$ is integrally closed.
		
		Let $z = x+y\alpha \in \tilde{J}$, with $x \in R$ and $y \in I$. Now $v'(x) = 2v(x)$ is even and $v'(y\alpha) = 2v(y)+m$ is odd, therefore $v'(x) \neq v'(y\alpha)$ and it results
		\[ v'(x) \geq \min\{v'(x),v'(y\alpha)\} = v'(z) \geq \tilde{j_1}. \]
		It follows that $x \in \tilde{J} \cap R = J$ and consequently $y\alpha = z-x \in \tilde{J}$.
		
		Now, if $\tilde{j_1}$ is even, let $z = x+y\alpha \in \tilde{J}$ such that $v'(z) = \tilde{j_1}$. From the previous observations it must be $\tilde{j_1} = v'(z) = v'(x) \geq j_1 \geq \tilde{j_1}$, hence $v'(x) = j_1 = \tilde{j_1}$ with $x \in J$.
		
		Since $J \subseteq \tilde{J}$ and $J \subseteq \overline{R}$, we have $v'(J) \subseteq v'(\tilde{J})$ and $v'(J) \subseteq v'(\overline{R})$, so $v'(J) \subseteq v'(\tilde{J}) \cap v'(\overline{R})$. Conversely let $z = x+y\alpha \in \mathcal{R}$ such that $v'(z) \in v'(\tilde{J}) \cap v'(\overline{R})$; $\tilde{J}$ integrally closed implies $z \in \tilde{J}$, for what we have proved so far $x \in J$. Since $v'(z) \in v'(\overline{R})$, it is even, thus $v'(z) = v'(x) \in v'(J)$.
		
		Now if $\tilde{j_1} < \min v'(I \alpha)$, assume by contradiction that $\tilde{j_1}$ is odd. Let $z = x+y\alpha \in \tilde{J}$ such that $v'(z) = \tilde{j_1}$, then $\tilde{j_1} = v'(z) = v'(y\alpha) \in v'(I\alpha)$, contradicting $\tilde{j_1} < \min v'(I\alpha)$.
	\end{proof}
	
	\begin{lemma}\label{I c C}
		If $I \subseteq C$ then $I \alpha \subsetneq C_{\mathcal{R}}$ and $\tilde{j_1}$ is even.
	\end{lemma}
	\begin{proof}
		Recalling that $v'(\mathcal{R}) = v(R) \Join^{v(b)}v(I)$, from \cite[Proposition 2.1]{d2013numerical} we obtain
		\begin{gather*}
			\min v'(C_{\mathcal{R}}) = 2\min v(I) + v(b)-1 = \min v'(I) + v'(\alpha) - 1 = \\
			= \min v'(I \alpha) -1 < \min v'(I \alpha).
		\end{gather*}
		Hence $\min v'(C_{\mathcal{R}}) < \min v'(I \alpha)$, so $I\alpha \subsetneq C_{\mathcal{R}}$ ($C_{\mathcal{R}}$ is integrally closed). Now, since $C_{\mathcal{R}} \subsetneq \tilde{J}$ and both $C_{\mathcal{R}}$ and $\tilde{J}$ are integrally closed, we have $\tilde{j_1} < \min v'(C_{\mathcal{R}}) < \min v'(I\alpha)$, from Proposition \ref{J int closed}, $\tilde{j_1}$ is even.
	\end{proof}
	
	\begin{lemma}\label{I c J}
		If $I \subseteq J$ then $I\alpha \subsetneq \tilde{J}$ and $\tilde{j_1}$ is even. Further, if $R$ is Arf and $I$ is integrally closed, then $JI\alpha \subseteq x \tilde{J}$, where $x \in J$ is such that $v'(x) = j_1 = \tilde{j_1}$.
	\end{lemma}
	\begin{proof}
		For every $i \in I$ it follows that $v'(i\alpha) = v'(i) + v'(\alpha) > v'(i)$, then $\min v'(I \alpha) > \min v'(I)$. Moreover $I \subseteq J \subseteq \tilde{J}$, hence
		\[ \min v'(I \alpha) > \min v'(I) \geq \min v'(J) \geq \tilde{j_1}. \]
		It follows that $I \alpha \subsetneq \tilde{J}$ ($\tilde{J}$ is integrally closed), and from Proposition \ref{J int closed} $\tilde{j_1}$ is even.
		
		Now suppose that $R$ is Arf and $I$ is integrally closed. The choice of $x \in J$ is allowed by Proposition \ref{J int closed}. Let $i \in I$ and $j \in J$, then $ij \in IJ \subseteq J^2 = xJ$ ($R$ is Arf, $J$ is integrally closed and $x$ is of minimum value in $J$), so there exists $j' \in J$ such that $ij = xj'$, furthermore
		\[ v'(j') = v'(ijx^{-1}) = v'(i) + v'(j)-v'(x) \geq v'(i), \]
		hence $j' \in I$ ($I$ is integrally closed). Finally $ji\alpha = xj'\alpha \in xI\alpha \subseteq x \tilde{J}$, therefore $JI\alpha \subseteq x\tilde{J}$.
	\end{proof}
	
	\begin{remark}
		Recall that $\min v(C)$ is the conductor of $v(R)$. Note that if $I$ is integrally closed, the condition $\min(v(I)) < \min v(C) = c(v(R))$ (similar to the one of Theorem \ref{characterization on ns}) is equivalent to $C \subsetneq I$.
	\end{remark}
	
	Recall that, if $R$ is Arf, then the multiplicity sequences of $R$ and of $v(R)$ coincides. In this case, we will denote the multiplicity sequence of $R$ (and of $v(R)$) with $(e_0,e_1,\ldots,e_{n-1},1,1,\ldots)$, $e_{n-1} \neq 1$.
	
	\begin{lemma}\label{J c I}
		If $J \subseteq I$, then $J\alpha \subseteq \tilde{J}$. Further, if $R$ is Arf, $I$ is integrally closed and $v(b) = e_0 = e_1 = \ldots = e_{n-1}$, then $J^2 \subseteq x\tilde{J}$ and $J^2\alpha \subseteq x\tilde{J}$, where $x \in \tilde{J}$ is an element of minimum value in $\tilde{J}$.
	\end{lemma}
	\begin{proof}
		If $j \in J$ then $j\alpha \in J\alpha \subseteq I\alpha \subseteq \mathcal{R}$, moreover $v'(j\alpha) \geq v'(j) \geq \tilde{j_1}$. It follows that $j\alpha \in \tilde{J}$, therefore $J\alpha \subseteq \tilde{J}$.
		
		Now suppose that $R$ is Arf, $I$ is integrally closed and $v(b) = e_0 = e_1 = \ldots = e_{n-1}$. If $\tilde{j_1}$ is even, from Proposition \ref{J int closed} we can choose $x \in J$; it follows that $J^2 = xJ \subseteq x\tilde{J}$ and $J^2\alpha = xJ\alpha \subseteq x \tilde{J}$.
		
		On the other hand, if $\tilde{j_1}$ is odd we can choose $y \in I$ such that $x = y\alpha \in \tilde{J}$. In this case it must be $C \subsetneq I$, otherwise if $I \subseteq C$ ($I$ is integrally closed), then from Lemma \ref{I c C} $\tilde{j_1}$ can not be odd. Therefore, since $C_{\mathcal{R}} \subsetneq \tilde{J}$ the value of $y\alpha$ is less than $\min v'(C_{\mathcal{R}})$, which is the conductor of $v'(\mathcal{R})$, from Remark \ref{remark multiples of m} we have $v'(y\alpha) = kv(b)$ for some $k \in \mathbb{N}$ odd. From Proposition \ref{J int closed} it follows that the minimum of $v'(J)$ is equal to $(k+1)v(b)$, hence
		\[ (k+1)v(b) = kv(b) + v(b) = v'(y\alpha) + v'(\alpha) = v'(yb) = j_1. \]
		Since $yb \in R$, $yb = x\alpha$ is an element of minimum value in $J$, therefore $J^2 = x\alpha J \subseteq x \tilde{J}$ and $J^2 \alpha = x\alpha^2J = xbJ \subseteq xJ \subseteq x\tilde{J}$.
	\end{proof}
	
	\begin{lemma}\label{J minus C}
		Suppose that $R$ is Arf, $I$ is integrally closed and $v(b) = e_0 = e_1 = \ldots = e_{n-1}$, and let $i \in I$. If $i\alpha \in \tilde{J} \setminus C_{\mathcal{R}}$ with $v'(i\alpha) > \tilde{j_1}$, then $i \in J$.
	\end{lemma}
	\begin{proof}
		Note that we can assume $C \subsetneq I$, otherwise if $I \subseteq C$ from Lemma \ref{I c C} $I \alpha \subseteq C_{\mathcal{R}}$, so $I \alpha \cap (\tilde{J} \setminus C_{\mathcal{R}}) = \emptyset$.
		
		In view of Theorem \ref{characterization on ns} and Theorem \ref{double valuation}, since $i\alpha \notin C_{\mathcal{R}}$ we have $v'(i\alpha) = kv(b)$ for some $k \in \mathbb{N}$ odd. Moreover, $v'(i\alpha) > \tilde{j_1}$, so $(k-1)v(b) \in v'(\tilde{J})$, therefore
		\[ v'(i) = v'(i\alpha) - v'(\alpha) = kv(b) - v(b) = (k-1)v(b) \in v'(\tilde{J}). \]
		Hence $v'(i) \geq \tilde{j_1}$ with $i \in I \subseteq R$, so $i \in \tilde{J} \cap R = J$.
	\end{proof}
	
	Now we are ready to prove Theorem \ref{characterization on rings}, so we no longer hold the assumptions made on $\tilde{J}$. 

	\begin{theorem}\label{characterization on rings}
		$\mathcal{R}$ is Arf if and only if $R$ is Arf with multiplicity sequence $(e_0,e_1,\ldots,e_{n-1},1,1,\ldots)$, $e_{n-1} \neq 1$, $I$ is integrally closed and if $C \subsetneq I$ then $v(b) = e_0 = e_1 = \ldots = e_{n-1}$.
	\end{theorem}
	\begin{proof}
		\emph{Necessity}. From Proposition \ref{R is Arf} $R$ is Arf, therefore the multiplicity sequences of $R$ and $v(R)$ coincides. Since $\mathcal{R}$ is Arf, $v(R) \Join^{v(b)}\nolinebreak v(I)$ is an Arf numerical semigroup, so from Theorem \ref{characterization on ns} $v(I)$ is integrally closed, equivalently $I$ is integrally closed. Furthermore, if $C \subsetneq I$, equivalently if $\min v(I) < c(v(R))$, then $v(b) = e_0 = e_1 = \ldots = e_{n-1}$. \\
		\emph{Sufficiency}. Let $\tilde{J}$ be an integrally closed ideal of $\mathcal{R}$. From Proposition \ref{stable in conductor} applied to $\tilde{J}$ and $\mathcal{R}$, if $\tilde{J} \subseteq C_{\mathcal{R}}$ then $\tilde{J}$ is stable, so suppose that $C_{\mathcal{R}} \subsetneq \tilde{J}$. We denote with $J = \tilde{J} \cap R$, from Proposition \ref{J int closed} $J$ is an integrally closed ideal of $R$, so it is stable. Let $x \in \tilde{J}$ be an element of minimum value, we want to prove that $x\tilde{J} = \tilde{J}^2$, the inclusion $x\tilde{J} \subseteq \tilde{J}^2$ is clear, so it is suffice to prove that $\tilde{J}^2 \subseteq x\tilde{J}$. Now the two ideals $I$ and $C$ of $R$ are both integrally closed, so one is contained in the other.
		
		If $I \subseteq C$ from Proposition \ref{J int closed} and Lemma \ref{I c C} we can choose $x \in J$ and it results $\tilde{J} \subseteq J + I \alpha \subseteq J + C_{\mathcal{R}}$. In the view of Proposition \ref{JC c xJ} applied to $\tilde{J}$ and $\mathcal{R}$ we obtain
		\[ \tilde{J}^2 \subseteq (J+C_{\mathcal{R}})^2 = J^2 + JC_{\mathcal{R}} + C_{\mathcal{R}}^2 \subseteq xJ + \tilde{J}C_{\mathcal{R}} + \tilde{J}C_{\mathcal{R}} \subseteq x\tilde{J}. \]
		
		If $C \subsetneq I$ in this case we have $v(b) = e_0 = e_1 = \ldots = e_{n-1}$. Again $I$ and $J$ are two integrally closed ideal of $R$ and we distinguish two cases.
		
		If $I \subseteq J$ then from Proposition \ref{J int closed} and Lemma \ref{I c J} we can choose $x \in J$ and we have $JI\alpha \subseteq x \tilde{J}$. It follows
		\[ \tilde{J}^2 \subseteq (J+I\alpha)^2 = J^2 + JI\alpha + I^2b \subseteq J^2 + x\tilde{J} + J^2 = xJ + x\tilde{J} \subseteq x\tilde{J}. \]
		If $J \subseteq I$ then let $x_1+y_1\alpha, x_2+y_2\alpha \in \tilde{J}$. For $k=1,2$ we distinguish three cases.
		\begin{enumerate}
		\item $v'(y_k \alpha) = \tilde{j_1}$. In this case set $y_k \alpha = x$.
		\item $y_k \alpha \in \tilde{J} \setminus C_{\mathcal{R}}$ with $v'(y_k \alpha) > \tilde{j_1}$. In this case, from Lemma \ref{J minus C}, $y_k \in J$.
		\item $y_k \alpha \in C_{\mathcal{R}}$.
		\end{enumerate}
		In view of Lemma \ref{J c I} and Proposition \ref{JC c xJ} we have to verify six different cases.
		\begin{itemize}
			\item[$(1,1)$]
			Set $y_1 \alpha = x$. Since $v'(y_1 \alpha) = v'(y_2 \alpha) = \tilde{j_1}$ and $\mathcal{R}$ is residually rational, there exists a unit $u \in \mathcal{R}$ such that $v'((y_2-uy_1)\alpha) > \tilde{j_1}$, thus for the element $(y_2-uy_1)\alpha$ we can proceed as in the cases $2$ or $3$:
			\begin{gather*}
			(x_1 +y_1 \alpha)(x_2 +y_2 \alpha) = x_1x_2 + x_1(y_2-uy_1)\alpha + y_1 \alpha(ux_1+x_2+y_2\alpha) \in \\
			\begin{split}
				\in J^2 + (y_2-uy_1) \alpha J + x \tilde{J} & \subseteq x J + (y_2-uy_1) \alpha J + x\tilde{J} \subseteq \\
				& \subseteq (y_2-uy_1) \alpha J + x\tilde{J}.
			\end{split}
			\end{gather*}
			Now if $(y_2-uy_1) \alpha \in \tilde{J} \setminus C_{\mathcal{R}}$, then from Lemma \ref{J minus C} $(y_2-uy_1) \in J$ so $(y_2-uy_1) \alpha J \subseteq \alpha J^2 \subseteq x \tilde{J}$. Otherwise if $(y_2-uy_1) \alpha \in C_{\mathcal{R}}$, then $(y_2-uy_1) \alpha J \subseteq C_{\mathcal{R}} \tilde{J} \subseteq x \tilde{J}$.
			
			\item[$(1,2)$] \begin{gather*}
			(x_1 +y_1 \alpha)(x_2 +y_2 \alpha) = x_1(x_2 +y_2 \alpha) + y_1 \alpha (x_2 +y_2 \alpha) \in \\
			\in J(J+J\alpha) + x \tilde{J} \subseteq J^2 + J^2 \alpha + x\tilde{J} \subseteq x\tilde{J}.
			\end{gather*}
			
			\item[$(1,3)$] \begin{gather*}
			(x_1 +y_1 \alpha)(x_2 +y_2 \alpha) = x_1(x_2 +y_2 \alpha) + y_1 \alpha (x_2 +y_2 \alpha) \in \\
			J(J+C_{\mathcal{R}}) + x\tilde{J} = J^2 + JC_{\mathcal{R}} + x\tilde{J} \subseteq x J + \tilde{J}C_{\mathcal{R}} + x\tilde{J} \subseteq x \tilde{J}.
			\end{gather*}
			
			\item[$(2,2)$] \begin{gather*}
			(x_1 +y_1 \alpha)(x_2 +y_2 \alpha) \in (J+J\alpha)^2 = J^2 + J^2\alpha + J^2b \subseteq \\
			\subseteq x\tilde{J} + x\tilde{J} + J^2\subseteq x\tilde{J}.
			\end{gather*}
			
			\item[$(2,3)$] \begin{gather*}
			(x_1 +y_1 \alpha)(x_2 +y_2 \alpha) \in (J+J\alpha)(J+C_{\mathcal{R}}) = \\
			= J^2 + JC_{\mathcal{R}} + J^2\alpha + J\alpha C_{\mathcal{R}} \subseteq x\tilde{J} + \tilde{J}C_{\mathcal{R}} + x\tilde{J} + \tilde{J}C_{\mathcal{R}} \subseteq x \tilde{J}.
			\end{gather*}
			
			\item[$(3,3)$] \begin{gather*}
			(x_1 +y_1 \alpha)(x_2 +y_2 \alpha) \in (J+C_{\mathcal{R}})^2 = J^2 + J C_{\mathcal{R}} + C_{\mathcal{R}}^2 \subseteq \\
			\subseteq x\tilde{J} + \tilde{J} C_{\mathcal{R}} + \tilde{J} C_{\mathcal{R}} \subseteq x\tilde{J}.
			\end{gather*}
		\end{itemize}
		This proves that $\tilde{J}^2 \subseteq x\tilde{J}$.
	\end{proof}
	\noindent
	\textbf{Acknowledgments}. I would like to thank Marco D'Anna for his constant help during the drafting of the paper and for his useful suggestions.
	
	\bibliographystyle{abbrv}
	\bibliography{Reference}
	
\end{document}